\documentclass[11pt,reqno]{amsart}

\usepackage{amsmath,amsthm,amssymb}
\usepackage{amssymb}

\usepackage{graphics,graphicx}
\usepackage{hyperref}
\usepackage[usenames, dvipsnames]{xcolor}

\usepackage{mathtools}
\mathtoolsset{showonlyrefs}

\definecolor{darkblue}{rgb}{0.0,0.0,0.3}
\hypersetup{colorlinks,breaklinks,
  linkcolor=darkblue,urlcolor=darkblue,
anchorcolor=darkblue,citecolor=darkblue}

\usepackage[square,sort,comma,numbers]{natbib}
\usepackage{fancyvrb}
\usepackage{enumitem}
\usepackage{verbatim}

\usepackage{booktabs}
\usepackage{threeparttable}

\theoremstyle{plain}
\newtheorem{theorem}{Theorem}[section]
\newtheorem*{theorem*}{Theorem}
\newtheorem{lemma}[theorem]{Lemma}

\newtheorem*{proposition*}{Proposition}
\newtheorem{corollary}[theorem]{Corollary}
\newtheorem*{corollary*}{Corollary}

\theoremstyle{definition}
\newtheorem{remark}[theorem]{Remark}

\numberwithin{equation}{section}

\renewcommand{\Im}{\operatorname{Im}}
\renewcommand{\Re}{\operatorname{Re}}

\DeclareMathOperator{\SL}{SL}

\title{Triple Correlation Sums of Coefficients of Cusp Forms}
\author[Hulse, Kuan, Lowry-Duda, Walker]{Thomas A. Hulse, Chan Ieong Kuan, David
Lowry-Duda, Alexander Walker}

\thanks{The authors thank Mehmet K{\i}ral and the Nesin Mathematics Village of
Turkey for providing a stimulating and relaxing collaborative
experience during this project's infancy. We are inspired both by their
mathematics and compassion.}
\thanks{The third author is supported by the Simons Collaboration in Arithmetic
Geometry, Number Theory, and Computation via the Simons Foundation grant
546235.}
\thanks{David also gratefully acknowledges support from EPSRC Programme Grant
EP/K034383/1 LMF:\ L-Functions and Modular Forms, as well as John Cremona,
Damiano Testa, and the rest of the number theory group at the University of
Warwick.}
\date{\today}

\begin{document}

\maketitle

\begin{abstract} We produce nontrivial asymptotic estimates for shifted
sums of the form $\sum a(h)b(m)c(2m-h)$, in which $a(n),b(n),c(n)$ are
un-normalized Fourier coefficients of holomorphic cusp forms. These results are
unconditional, but we demonstrate how to strengthen them under the Riemann
Hypothesis. As an application, we show that there are infinitely many three term
arithmetic progressions $n-h, n, n+h$ such that $a(n-h)a(n)a(n+h) \neq 0$.
\end{abstract}

\section{Introduction}

Convolution sums formed from coefficients of modular forms have frequent
applications throughout number theory. Let $d(n)$ denote the divisor function
and let $r_d(n)$denote the number of representations of $n$ as a sum of $d$
squares. Correlation sums of the approximate forms
\begin{equation}
  \sum_{n \leq X} d(n) d(n + h)
  \quad \text {or} \quad
  \sum_{\substack{n \leq X \\ h \leq Y}} r_d(n) r_d(n + h)
\end{equation}
appear in off-diagonal terms for fourth moment estimates of the Riemann zeta
function (as in~\cite{HB79}), for second moment estimates of the Gauss circle
problem (as in~\cite{ivic2003note}), and for second moment estimates of the
$d$-dimensional Gauss circle problem (as in~\cite{HKLW4}).
These correlation sums are well-studied and many techniques have been developed
to understand their asymptotic behavior.

Triple correlation sums of the form
\begin{equation}
  \sum a(n - h) b(n) c(n + h)
\end{equation}
can also exhibit distinguished behavior, though they are far less understood.
Blomer~\cite{blomer2017triple} used the spectral theory of automorphic forms to
produce asymptotics for partially smoothed triple correlation sums of the form
\begin{equation}
  \sum_{h} W\big( \tfrac{h}{H} \big)
  \sum_{N \leq n \leq 2N}
  d(n - h) a(n) d(n + h),
\end{equation}
where $a(n)$ is any sequence of complex numbers and $W$ is a smooth bump
function.
Lin~\cite{lin2018triple} built on Blomer's analysis to establish similar bounds
for triple correlation sums of Fourier coefficients of cusp forms. In particular,
letting $A(n)$ denote the normalized Fourier coefficients of a Hecke eigenform,
Lin proves that
\begin{equation}
  \sum_h W \big( \tfrac{h}{H} \big)
  \sum_{N \leq n \leq 2N}
  A(n - h) A(n) A(n + h)
  \ll
  N^\epsilon \min \Big( NH, \frac{N^2}{H^{1/2}} \Big).
\end{equation}
This estimate is nontrivial when $H \geq N^{\frac{2}{3} + \epsilon}$ and
otherwise matches the trivial bound from bounding by the length of the sum.
Singh~\cite{singh2018double} used the circle method of Heath-Brown to prove that
\begin{equation}
  \frac{1}{H} \sum_h W_1 \big( \tfrac{h}{H} \big)
  \sum_{n \leq N} A(n) B(n + h) C(n + 2h) W_2 \big( \tfrac{n}{N} \big)
  \ll
  N^{1 - \delta}
\end{equation}
when $H \gg N^{\frac{1}{2} + \epsilon}$ for some $\delta > 0$, where $A(n)$,
$B(n)$, and $C(n)$ are normalized coefficients of holomorphic cusp forms or
Maass eigenforms on $\SL(2, \mathbb{Z})$. Singh's result allows a more
concentrated sum in $H$ at the cost of a slightly different form of triple
correlation.

In this paper, we consider yet another form of triple correlation
between coefficients of cusp forms. Namely, we consider triple correlation sums of the
form
\begin{equation}
  \sum_{m, h}
  a(h) b(m) c(2m - h) e^{-m/X} e^{-h/Y}
\end{equation}
in which $a(n)$, $b(n)$, and $c(n)$ denote \emph{non-normalized} coefficients of a
holomorphic cuspidal Hecke eigenforms $f_1$, $f_2$, and $f_3$, respectively,
each of even weight $k$, level $N$, and trivial nebentypus.

To attain heuristic estimates for sums of this form, note that when $Y=O(X)$,
\begin{align*}
  \sum_{\substack{m \leq X \\ h \leq Y}}
  &a(h)b(m)c(2m-h) \\[-20pt]
  \vspace{-20em} &\ll
  \sum_{\substack{m \leq X \\ h \leq Y}}
  h^{\frac{k-1}{2} + \epsilon}m^{k-1 + \epsilon}
  \ll
  \begin{cases}
    X^{k-1 + 1 + \epsilon} Y^{\frac{k-1}{2} + 1 + \epsilon} & \text{naively} \\[5pt]
    X^{k-1 + \frac{1}{2} + \epsilon} Y^{\frac{k-1}{2} + \frac{1}{2} + \epsilon}
      & \parbox{10em}{double square-root \\ cancellation}
  \end{cases}
\end{align*}
The naive estimate follows from termwise application of Deligne's bound for each coefficient and a
bound by absolute values.
The second estimate follows from assuming that
there is square-root type cancellation in both the $m$ and $h$ sums. This would
occur if the $m$ and $h$ sums experience independent random sign changes,
but whether or not that occurs is unknown.

Our first theorem for these correlation sums is that square-root cancellation
occurs in both the $m$ and $h$ sums at all scales.

\begin{theorem}
  Let $a(\cdot)$, $b(\cdot)$, and $c(\cdot)$ denote the coefficients of holomorphic
  cuspidal eigenforms of weight $k$, level $N$, and trivial nebentypus.
  Then for any $\epsilon > 0$,
  \begin{equation*}
    \sum_{m, h \geq 1}
    a(h)b(m)c(2m-h) e^{-m/X} e^{-h/Y}
    \ll
    X^{k-1 + \Theta + \frac{1}{2} + \epsilon}
    Y^{\frac{k-1}{2} -\Theta + \frac{1}{2} + \epsilon}.
  \end{equation*}
  Here $\Theta < 7/64$ refers to the best bound towards Selberg's Eigenvalue
  Conjecture.
\end{theorem}

Under the assumption of the Riemann Hypothesis, we can prove that there is
square-root cancellation in $X$ and $3/4$-type cancellation in $Y$.

\begin{theorem}
  Assume the Riemann Hypothesis.
  Then with the same notation as above and for any $\epsilon > 0$, we have
  \begin{equation*}
    \sum_{m, h \geq 1}
    a(h)b(m)c(2m-h) e^{-m/X} e^{-h/Y}
    \ll
    X^{k-1 + \Theta + \frac{1}{2} + \epsilon}
    Y^{\frac{k-1}{2} -\Theta + \frac{1}{4} + \epsilon}.
  \end{equation*}
\end{theorem}

There are famous sums with conjectured $3/4$-type cancellation, including the
Gauss Circle problem and the Dirichlet Divisor problem. The work of
Chandrasekharan and Narasimhan~\cite{Chand} implies that
\begin{equation*}
  \frac{1}{X} \int_1^X \sum_{n \leq t} a(n) dt
  \ll
  X^{\frac{k-1}{2} + \frac{1}{4} + \epsilon},
\end{equation*}
indicating that the sums of coefficients of cusp forms experience $3/4$-type
cancellation on average. It appears that this large degree of cancellation
carries into the $h$-sum. We will see below that we do not expect
more-than-square-root cancellation in the $m$-sum.

We use the spectral theory of modular forms to prove these theorems. In the
analysis, several lines of spectral poles appear. Applying a theorem
from~\cite{dldomega}, it follows that any one of these poles
guarantees non-vanishing of infinitely many triples
$a(h)b(n)c(2n-h)$.

\begin{theorem}\label{thm:omega}
  Maintaining the same notation as the two previous theorems, fix $0 < \alpha <
  1$. Suppose there exists a non-constant Maass form $\mu_j$ on $\Gamma_0(2N)$
  with Laplacian eigenvalue $\lambda > \frac{1}{4}$ such that $\langle
  f_2(2z)\overline{f_3(z)}y^k, \mu_j \rangle \neq 0$. Then
  \begin{equation*}
  \Bigg\lvert
    \sum_{m \leq X} \sum_{h \leq 2X}
    \frac{a(h)b(m)c(2m-h)}
         {h^{\frac{k}{2} + \frac{3}{2}}}
  \Bigg\rvert
  = \Omega(X^{k - \frac{1}{2}}).
  \end{equation*}

  Therefore infinitely many terms of the dual-sequence
  \begin{equation*}
    {\{a(h)b(m)c(2m-h)\}}_{m,h \in \mathbb{N}}
  \end{equation*}
  are nonzero.
\end{theorem}

One interpretation of the preceding theorem is that the $X$ sum often
has no more than square-root cancellation. The power of $h$ appearing in the
theorem is mostly technical, and does not affect this interpretation.

Applied to the case when $f_1 = f_2 = f_3 = f$, we get the following corollary.
\begin{corollary}
  Suppose that there exists a non-constant Maass form $\mu_j$ on $\Gamma_0(2N)$
  with Laplacian eigenvalue $\lambda > \frac{1}{4}$
  such that $\langle f(2z) \overline{f(z)} y^k, \mu_j \rangle \neq 0$.

  Then there are infinitely many three-term arithmetic progressions $n-h, n, n+h$
  such that
  \begin{equation*}
    a(n-h)a(n)a(n+h) \neq 0.
  \end{equation*}
\end{corollary}

\begin{remark}
The results above may be further generalized to concern triples of modular forms
$f_i$ which do not necessarily have the same level, weight, or nebentypus. The
restrictions we impose are used to simplify the exposition of the proof;
loosening them would not significantly alter the overall argument.
\end{remark}

\section*{Motivation from the Congruent Number Problem}

Our initial motivation to understand sums of this form came from the congruent
number problem. Recall that a \emph{congruent number} is an integer which
appears as the area of a a right triangle with rational-length sides. The
congruent number problem is the classification problem of determining which
integers are congruent. It is a well-studied classical problem;
see~\cite{conrad} for a nice survey.

There is a well-known correspondence between three-term arithmetic progressions
of squares and congruent numbers, in which the common difference in the progression is a
congruent number. Let $\theta(z) = \sum_{n \in \mathbb{Z}} e(n^2 z) = \sum_{n
\geq 0} r_1(n) e(nz)$ denote the classical theta function, where $e(z) = e^{2 \pi i z}$.
Here, $r_1(n)$ is essentially twice the square-indicator function, except that
$r_1(0) = 1$. Then if $r_1(h)r_1(m)r_1(2m-h) \neq 0$, the triple $(h, m,
2m-h)$ is a three-term arithmetic progression of squares and $m-h$ is congruent.

Understanding sums of the shape
\begin{equation}\label{eq:ifweshadowshaveoffended}
  \sum r_1(h)r_1(m)r_1(2m-h)
\end{equation}
would open up new approaches to understanding the distribution of congruent
numbers. Sums of the above shape can be attained from the primary sum
\begin{equation*}
  \sum_{m,h =1}^\infty r_1(m-h)r_1(m)r_1(m+h)r_1(th)
\end{equation*}
where $t$ is a square-free integer, studied by the authors
in~\cite{hkldw_congruent}, after summing over all such $t$, although the
implicit dependence of the error term on $t$ prevents detailed analysis.

Heuristically, by replacing holomorphic cusp forms with the classical theta function
$\theta(z) = \sum e(n^2z)$, where $e(z) = e^{2\pi i z}$, the main results of this paper would
describe sums of the shape~\eqref{eq:ifweshadowshaveoffended}. Furthermore, one
would attain meromorphic continuation for the series
\begin{equation*}
  \widetilde{D}(s,w)
  =
  \sum_{m,h\geq 1}^\infty
  \frac{r_1(h)r_1(m)r_1(2m - h)}
       {m^{s - \frac{1}{2}} h^w}.
\end{equation*}
This series would provide additional tools to investigate the asymptotic
behavior of congruent numbers.

There are significant challenges to carrying out this heuristic: it's necessary
to work in higher level with additional cusps; $\theta(z)$ is half-integral
weight; and perhaps most significantly, $\theta(z)$ is not cuspidal. The authors
hope to continue this investigation in later work.

\section{Methodology and Notation}

Let $f_1(z) = \sum a(n) e(nz)$ be a holomorphic cuspidal Hecke eigenform of
weight $k$ on $\Gamma_0(N)$ with trivial nebentypus and real coefficients;
similarly, define $f_2(z)$ and $f_3(z)$ with respective coefficients $b(n)$ and
$c(n)$. We will investigate the meromorphic continuation of the shifted multiple
Dirichlet series
\begin{equation}\label{def:Dsw}
  D(s,w)
  :=
  \sum_{m,h \geq 1} \frac{a(h)b(m)c(2m-h)}{m^{s+k-1} h^w},
\end{equation}
defined initially for $\Re s, \Re w$ sufficiently large. Ultimately, we will
show that this double Dirichlet series has meromorphic continuation to
$\mathbb{C}^2$ and has polynomial growth in vertical strips (away from poles).

Let $V(z)$ denote the product $V(z) = y^k f_2(2z) \overline{f_3(z)}$. Note
that $f_2(2z)$ is a holomorphic cusp form of weight $k$ and level $2N$.
Define the level $2N$ Poincar\'e series as
\begin{equation*}
  P_h(z, s; 2N)
  :=
  \sum_{\gamma \in \Gamma_\infty \backslash \Gamma_0(2N)}
  \Im{(\gamma z)}^s e(h\gamma z)
\end{equation*}
and note that this converges locally uniformly on the upper-half plane
$\mathbb{H}$ and belongs to $L^2(\Gamma_0(2N) \backslash \mathbb{H})$.


Then the classical unfolding computation shows that the Petersson inner product
$\langle V, P_h(\cdot, \overline{s}; 2N) \rangle$ gives the double correlation
sum
\begin{align*}
  \langle V, P_h(\cdot, \overline{s}; 2N) \rangle
  &=
  \frac{\Gamma(s + k - 1)}{{(8 \pi)}^{s + k - 1}}
  \sum_{m = 1}^\infty \frac{b(m) c(2 m - h)}{m^{s + k - 1}}.
\end{align*}
We recognize $D(s, w)$ as the sum
\begin{equation}\label{eq:Dsw_inner_product}
  D(s, w)
  =
  \frac{{(8\pi)}^{s + k - 1}}{\Gamma(s + k - 1)}
  \sum_{h \geq 1}
  \frac{a(h) \langle V, P_h(\cdot, \overline{s}; 2N) \rangle}{h^w},
\end{equation}
which converges absolutely for $\Re s$ and $\Re w$ sufficiently large.
In Section~\ref{sec:spectral_expansion}, we replace $P_h$ with its spectral
expansion to obtain an alternate description of $D(s,w)$.

In Section~\ref{sec:Dsw}, we use this spectral expansion to study the
meromorphic properties of $D(s,w)$. The broad
methodology of this section is similar to classical ideas of Selberg, more
recently refined in the appendix to Sarnak~\cite{Sarnak}, the work of Blomer and
Harcos~\cite{blomer2008spectral}, and the work of Hoffstein, the first author,
and Reznikov~\cite{Jeff}. The observation that it is possible to multiply by
$a(h)h^{-w}$ and still make sense of the resulting spectral decomposition has
been observed by Hoffstein, and is used in the first author's thesis.

The main result of Section~\ref{sec:Dsw} is to show that $D(s, w)$ has
meromorphic continuation to $\mathbb{C}^2$ and to describe the nature of the
leading poles. With this description, the remainder of the paper is very
straightforward. The final section then shows that $D(s, w)$ has polynomial
growth in vertical strips and proves the main theorems.

We note that it is possible to obtain various weighted averages of triple
correlation sums from the meromorphic properties of $D(s, w)$ by adapting the
methods that yield the main results in this paper; these might be applied to
yield interesting results in the future.

\section{Spectral Expansion}\label{sec:spectral_expansion}

In this section, we use the spectral expansion of $P_h(z, s; 2N)$
to rewrite $D(s, w)$ in a way that exposes its meromorphic properties.

By Selberg's Spectral Theorem (as in~\cite[Theorem~15.5]{IK}),
the Poincar\'{e} series $P_h(z, s; 2N)$ has a spectral expansion of the form
\begin{equation}\label{eq:spectral_expansion_poincare}
  \begin{split}
    P_h(z, s; 2N) &= \sum_j \langle P_h(\cdot,s; 2N),\mu_j \rangle \mu_j(z)
    \\
    &\,\, +
    \sum_\mathfrak{a} \frac{1}{4\pi}\!\int\limits_{-\infty}^\infty \!
    \langle P_h(\cdot,s; 2N),E_\mathfrak{a}(\cdot,\tfrac{1}{2}+it; 2N)\rangle
    E_\mathfrak{a}(z,\tfrac{1}{2}+it; 2N)dt,
  \end{split}
\end{equation}
where $\mathfrak{a}$ ranges over the cusps of $\Gamma_0(2N) \backslash
\mathbb{H}$; $E_\mathfrak{a}$ denotes Eisenstein series associated to
$\mathfrak{a}$; and $\{\mu_j\}$ denotes an orthonormal basis of the residual and
cuspidal spaces, consisting of the constant form $\mu_0$ and of Hecke-Maass
forms $\mu_j$ for $L^2(\Gamma_0(2N)\backslash \mathbb{H})$ with associated
types $\frac{1}{2}+it_j$.
The inner product of the Poincar\'{e} series against the constant form $\mu_0$
vanishes, so we omit further mention of it.


Here, $E_\mathfrak{a}(z,s;2N)$ is the Eisenstein series of level $2N$ defined by
\begin{equation}
  E_\mathfrak{a}(z,s;2N)
  =
  \sum_{\gamma \in \Gamma_\mathfrak{a} \backslash \Gamma_0(2N)}
  \Im{(\sigma_{\mathfrak{a}}^{-1}\gamma z)}^s,
\end{equation}
where $\Gamma_\mathfrak{a} \subset \Gamma_0(2N)$ is the stabilizer of the cusp
$\mathfrak{a}$ and $\sigma_\mathfrak{a} \in \mathrm{PSL}_2(\mathbb{R})$
satisfies $\sigma_\mathfrak{a} \infty = \mathfrak{a}$, induces an isomorphism
$\Gamma_\mathfrak{a} \cong \Gamma_\infty$ via conjugation, and is unique up to
right translation. We refer to the sum over $j$ as the ``discrete part of the
spectrum'' and the sum of integrals of Eisenstein series as the ``continuous
part of the spectrum.''

By replacing $P_h$ with its spectral expansion in $\langle V, P_h\rangle$, we
obtain the expansion
\begin{align}\label{eq:VPh_inner_product}
  \langle V(z)&, P_h(z,\overline{s};2N) \rangle
  =
  \sum_j \overline{\langle P_h(\cdot,\overline{s};2N), \mu_j\rangle}
  \langle V,\mu_j \rangle
  \\
  &+
  \sum_{\mathfrak{a}} \frac{1}{4\pi}
  \int_{-\infty}^\infty
  \overline{\langle
    P_h(\cdot,\overline{s};2N), E_{\mathfrak{a}}(\cdot,\tfrac{1}{2}+it;2N)
  \rangle}
  \langle V, E_{\mathfrak{a}}(\cdot,\tfrac{1}{2}+it;2N) \rangle dt.
\end{align}

The Fourier expansions of the Maass forms and Eisenstein series are known and
can be used to understand the inner products against the Poincar\'{e} series.
The Maass forms have Fourier expansions of the form
\begin{equation}
  \mu_j(z)
  =
  \sqrt{y} \sum_{\lvert m\rvert \neq 0} \rho_j(m)
  K_{it_j}(2\pi \lvert m \rvert y) e(mx),
\end{equation}
where $K_{it_j}$ is a $K$-Bessel function.
For each Maass form $\mu_j$, there is a constant $\rho_j(1)$ such that for each
prime $p$ with $\gcd(p, 2N) = 1$, the coefficient $\rho_j(p)$ can be written
$\rho_j(p) = \rho_j(1) \lambda_j(p)$, where $\lambda_j(p)$ is the eigenvalue of
the $p$-th Hecke operator.
In level $1$, this common constant is the first coefficient $\rho_j(1)$.
By a minor abuse of notation, we continue to use the notation $\rho_j(1)$ even
though the $m=1$ Fourier coefficient might be zero. Thus we write $\rho_j(h) = \rho_j(1) \lambda_j(h)$.

The Eisenstein series have Fourier expansions
\begin{align*}
  E_\mathfrak{a}(z,s;2N)
  &=
  \delta_{\mathfrak{a},\infty}y^{s}
  +
  \frac{\sqrt{\pi}\Gamma(s-\frac{1}{2})\rho_{\mathfrak{a}}(s,0)y^{1-s}}
       {\Gamma(s)}
  \\
  &\quad
  + \frac{2\pi^s \sqrt{y}}{\Gamma(s)}
  \sum_{m \neq 0}
  \lvert m \rvert^{s-\frac{1}{2}}
  \rho_{\mathfrak{a}}(s,m) K_{s-\frac{1}{2}}(2\pi \lvert m \rvert y) e(mx)
\end{align*}
with computable coefficients $\rho_\mathfrak{a}(s, m)$.

With these expansions, one can explicitly compute the inner products as
\begin{align}
  \overline{\langle
    P_h(\cdot,\overline{s};2N), E_{\mathfrak{a}}(\cdot,\tfrac{1}{2}+it;2N)
  \rangle}
  &=
  \frac{\rho_{\mathfrak{a}}(\frac{1}{2}+it,h)
    \Gamma(s-\frac{1}{2}+it)
    \Gamma(s-\frac{1}{2}-it)
  }
  {4^{s-1} \pi^{s-\frac{3}{2}-it}
    h^{s-\frac{1}{2}-it}
    \Gamma(s)
    \Gamma(\frac{1}{2}+it)
  };
  \\
  \overline{\langle P_h(\cdot,\overline{s};2N), \mu_j\rangle}
  &=
  \frac{\rho_j(h) \sqrt{\pi}}
    {{(4\pi h)}^{s-\frac{1}{2}}}
  \frac{\Gamma(s-\frac{1}{2}+it_j) \Gamma(s-\frac{1}{2}-it_j)}{\Gamma(s)}.
\end{align}
This computation is another application of unfolding, in which one uses the
integral identity found in~\cite[6.621(3)]{GradshteynRyzhik07} to understand the
integrals involving $K$-Bessel functions.

Substituting these expressions into~\eqref{eq:VPh_inner_product} gives the
following spectral expansion.
\begin{lemma}[Spectral expansion]\label{lem:VPh_spectral}
  The inner product $\langle V, P_h(\cdot, \overline{s}; 2N) \rangle$ has the
  spectral expansion
\begin{equation}
\begin{split}
  &\langle V(z), P_h(z,\overline{s};2N) \rangle
  = \sum_j
  \frac{\rho_j(h)\sqrt{\pi}}{{(4\pi h)}^{s-\frac{1}{2}}}
  \frac{\Gamma(s-\frac{1}{2}+it_j)\Gamma(s-\frac{1}{2}-it_j)}{\Gamma(s)}
  \langle V,\mu_j \rangle
  \\
  &\quad +
  \sum_{\mathfrak{a}}
  \! \int_{-\infty}^\infty
  \frac{\rho_{\mathfrak{a}}(\frac{1}{2}+it,h)}
       {4^{s}(\pi h)^{s-\frac{1}{2}-it}}
  \frac{\Gamma(s-\frac{1}{2}+it)\Gamma(s-\frac{1}{2}-it)}
       {\Gamma(s)\Gamma(\frac{1}{2}+it)}
  \langle V, E_{\mathfrak{a}}(\cdot,\tfrac{1}{2}+it;2N) \rangle dt
\end{split}
\end{equation}
  for $\Re s$ sufficiently large.
  We refer to the sum indexed by $j$ as the ``discrete part'' and the sum
  indexed by the cusps $\mathfrak{a}$ as the ``continuous part'' of the spectral
  expansion.
\end{lemma}

\section{The double sum $D(s, w)$}\label{sec:Dsw}

To study the meromorphic continuation for $D(s,w)$, the double Dirichlet
series given by
\begin{equation}
  \sum_{m,h\geq 1}^\infty \frac{a(h)b(m)c(2m-h)}{m^{s+k-1}h^w}
  =
  \frac{{(8 \pi)}^{s+k-1}}{\Gamma(s+k-1)}
  \sum_{h \geq 1}
  \frac{a(h) \langle V,P_h(\cdot,\overline{s};2N)\rangle}{h^w},
\end{equation}
we substitute the inner products with the spectral expansion obtained in
Lemma~\ref{lem:VPh_spectral}.
We split our analysis into two parts based on the natural subdivision of
$\langle V, P_h \rangle$ into discrete and continuous spectral terms.
We also discuss the convergence of each part in turn.

\subsection{Discrete Spectrum}

The discrete component of $D(s, w)$ is obtained from the discrete part of the
spectral expansion in Lemma~\ref{lem:VPh_spectral} upon multiplying by
$a(h)(8\pi)^{s+k-1}/\big(h^{w}\Gamma(s+k-1)\big)$ and summing over $h$.
After simplification, the discrete component is
\begin{equation}\label{eq:discrete_ref}
  2^{s-2}(8 \pi)^k
  \sum_j
  \frac{\Gamma(s - \frac{1}{2} + it_j)\Gamma(s - \frac{1}{2} - it_j)}
       {\Gamma(s)\Gamma(s+k-1)}
  \langle V, \mu_j \rangle \rho_j(1)
  \sum_{h \geq 1} \frac{a(h) \lambda_j(h)}{h^{s + w - \frac{1}{2}}}.
\end{equation}
In simplifying this expression, we have exchanged the order of summation; this
needs justification.

It is clear that the $h$-sum converges absolutely for $\Re (s+w)$ sufficiently
large. The behavior of $\langle V, \mu_j \rangle \rho_j(1)$ is of polynomial
growth in $\lvert t_j \rvert$ on average. In particular,
Reznikov's appendix to~\cite{Jeff} prove the following lemma.
\begin{lemma}{(Reznikov's appendix to~\cite{Jeff})}\label{lem:jeff_bound}
  Suppose $f$ and $g$ are two weight $k$ cuspidal modular forms on the
  congruence subgroup $\Gamma_0(N)$. Then for any $\epsilon > 0$,
  \begin{equation}
    \sum_{\lvert t_j \rvert \sim T}
    \rho_j(1) \langle f \overline{g} \Im{(\cdot)}^k, \mu_j \rangle
    \ll
    T^{k + 1 + \epsilon}.
  \end{equation}
\end{lemma}
For any $s$ away from poles, Stirling's approximation shows that the gamma
functions give exponential decay in $\lvert t_j \rvert$, and thus the sum over
$j$ converges locally normally. Thus the double sum converges absolutely, and
the sums can be reordered.

\subsection{The $h$ sum}
We can recognize the $h$-sum,
\begin{equation}
  \rho_j(1) \sum_{h \geq 1} \frac{a(h) \lambda_j(h)}{h^s},
\end{equation}
as a Rankin--Selberg convolution 
$L$-function which is obtained by unfolding the inner product of the form
$\langle \mu_j \Im{(\cdot)}^{k/2} \overline{f_1}, E \rangle$, where $E$ is an
appropriately chosen Eisenstein series.

For this application, we use the weight $k$ Eisenstein series
\begin{equation}
  E_\infty^k(z,w; 2N)
  :=
  \sum_{\gamma \in \Gamma_\infty \backslash \Gamma_0(2N)}
  \Im{(\gamma z)}^w J{(\gamma, z)}^{-k},
\end{equation}
where $J(\gamma, z) = j(\gamma, z) / \lvert j(\gamma, z) \rvert$ and
$j(\gamma, z) = (cz+d)$ is our automorphic multiplier.
Then another unfolding computation shows that
\begin{equation}
\begin{split}
  &\langle
    \mu_j \Im{(\cdot)}^{k/2} \overline{f_1}, E_\infty^k (\cdot, \overline{s}; 2N)
  \rangle
  \\
  &\quad =
  \frac{\sqrt \pi}{{(4\pi)}^{s + \frac{k}{2} - \frac{1}{2}}}
  \frac{\Gamma(s + \frac{k}{2} - \frac{1}{2} - it_j)
        \Gamma(s + \frac{k}{2} - \frac{1}{2} + it_j)}
       {\Gamma(s + \frac{k}{2})}
  \sum_{h \geq 1}
  \frac{a(h) \rho_j(h)}{h^{s + \frac{k}{2} - \frac{1}{2}}}.
\end{split}
\end{equation}
Let $\zeta^{(2N)}(s)$ denote the completed zeta function with the Euler factors
corresponding to divisors of $2N$ omitted. Then the completed Eisenstein series
$\zeta^{(2N)}(2s)E_\infty^k(z, s; 2N)$ has a functional equation of the shape $s
\mapsto 1-s$ and poles at most at $s = 1$ and $0$. We therefore define the
Rankin-Selberg convolution $L$-function
$L(s, \mu_j \otimes \overline{f_1})$ as
\begin{equation}
  L(s, \mu_j \otimes \overline{f_1})
  =
  \zeta^{(2N)}(2s)
  \sum_{h \geq 1}
  \frac{a(h) \rho_j(h)}{h^{s + \frac{k}{2} - \frac{1}{2}}}.
\end{equation}
This $L$-function can be completed and satisfies the functional equation
\begin{align}
  \Lambda(s, \mu_j \otimes \overline{f_1})
  &:=
  L(s, \mu_j \otimes \overline{f_1})
  \frac{\Gamma(s)
        \Gamma(s + \frac{k}{2} - \frac{1}{2} - it_j)
        \Gamma(s + \frac{k}{2} - \frac{1}{2} + it_j)}
       {{4}^{s + \frac{k}{2} - 1} \pi^{2s+\frac{k}{2}-1} \;
        \Gamma(s + \frac{k}{2})}
  \\
  &= \Lambda(1-s, {\mu_j} \otimes {\overline{f_1}}).
  \label{equation:discrete_functional_equation}
\end{align}
Further, the meromorphic behavior of the Eisenstein series guarantees that
the completed $L$-function has poles at most at $s = 0$ and $s = 1$.

In this application, we can rewrite the $h$ sum as
\begin{equation}
  \sum_{h \geq 1} \frac{a(h) \rho_j(h)}{h^{s+w-\frac{1}{2}}}
  =
  \frac{L(s + w - \frac{k}{2}, \mu_j \otimes \overline{f_1})}
       {\zeta^{(2N)}(2s + 2w - k)}.
\end{equation}
As a result, we have the following lemma.

\begin{lemma}\label{lem:disc}
The discrete component of $D(s, w)$ in~\eqref{eq:discrete_ref} can be rewritten
as
\begin{equation}
\begin{split}
 \label{eq:discrete_spectrum2}
  \frac{{(8 \pi)}^{s + k - 1}}{\Gamma(s+k-1)}
  \sum_j &\frac{\sqrt \pi}{{(4\pi)}^{s - \frac{1}{2}}}
  \frac{\Gamma(s - \frac{1}{2} + it_j)\Gamma(s - \frac{1}{2} - it_j)}{\Gamma(s)}
  \\
  &\qquad  \times \langle V, \mu_j \rangle
  \frac{%
    L(s + w - \frac{k}{2}, \mu_j \otimes \overline{f_1})
  }{\zeta^{(2N)}(2s + 2w - k)}.
\end{split}
\end{equation}
Furthermore, the discrete component has meromorphic continuation to $\mathbb{C}^2$.
\end{lemma}

We note that for $s$ and $w$ in any compact set away from poles, the gamma
functions give exponential decay in $\lvert t_j \rvert$, giving locally normal
convergence. There are potential poles at $2s + 2w - k = \rho$ where $\rho$ is a
zero of the zeta function. There are also potential poles when $s - \frac{1}{2}
\pm it_j = -n$ for $n \in \mathbb{Z}_{\geq 0}$ from the gamma functions. In
light of the Selberg Eigenvalue Conjecture, we expect that the leading poles
of the latter type in $s$ form an infinite family with the same real part;
therefore later analysis will not give a consistent asymptotic leading term
coming from the discrete spectrum.

\begin{remark}
  When $w$ is a nonpositive integer, the poles from
  $s - \frac{1}{2} \pm it_j = -n$ do not occur. This can be observed by
  rewriting the discrete spectrum in terms of
  $\Lambda(\cdot, \mu_j \otimes \overline{f_1})$.
  This reveals pairs of gamma factors of the shape
  $\Gamma(s - \frac{1}{2} \pm it_j) / \Gamma(s + w - \frac{1}{2} \pm it_j)$,
  indicating the cancellation.
\end{remark}

\subsection{Continuous Spectrum}

The continuous component of $D(s, w)$ is obtained from the continuous part of
the spectral expansion in Lemma~\ref{lem:VPh_spectral} after multiplying by
$a(h){(8\pi)}^{s+k-1}/\big(h^{w}\Gamma(s+k-1)\big)$ and summing over $h$.
After simplification, the continuous component is
\begin{equation}
\begin{split}\label{eq:continuous_sum_ref}
  \frac{2^{s + 3k - 3} \pi^{k - \frac{1}{2}}} {\Gamma(s + k - 1)}
  \sum_{\mathfrak{a}} &\int_{-\infty}^\infty
  \pi^{it}
  \frac{\Gamma(s - \frac{1}{2} + it) \Gamma(s - \frac{1}{2} - it)}
       {\Gamma(s) \Gamma(\frac{1}{2} + it)}
  \left\langle
    V, E_\mathfrak{a}(\cdot, \tfrac{1}{2} + it; 2N)
  \right\rangle
  \\
  &\times
  \sum_{h \geq 1}
  \frac{a(h) \rho_{\mathfrak{a}}(\frac{1}{2} + it, h)}
       {h^{w + s - \frac{1}{2} - it}}
  \, dt.
\end{split}
\end{equation}
Classical bounds on $a(h)$ and Lemma~3.4 of~\cite{Blomer} imply that
$a(h)\rho_\mathfrak{a}(1/2 + it, h)$ has at most mild polynomial growth in $t$ and $h$.
The growth in $t$ is analogous to the grown in $t_j$ in the discrete spectrum.
The following lemma follows from Stirling's formula and the same result in
Reznikov's appendix to~\cite{Jeff} that Lemma~\ref{lem:jeff_bound} is derived
from.

\begin{lemma}{(Reznikov's Appendix to~\cite{Jeff})}\label{lem:jeff_bound_cont}
  With the notation above, we have the bound
  \begin{equation}
    \int_{-T}^T \frac{%
      \lvert \langle V, E_\mathfrak{a}(\cdot, \frac{1}{2} + it; 2N) \rangle \rvert
    }{%
      \lvert \Gamma(\frac{1}{2} + it) \rvert
    }
    dt \ll T^{1 + k + \epsilon}.
  \end{equation}
\end{lemma}

It follows from Stirling's formula and Proposition~4.1 of~\cite{Jeff} that the
integral has exponential decay in $t$. Thus for $\Re (s + w)$ sufficiently
large, this converges absolutely.

As with the discrete spectrum, we will recognize the sum over $h$ as a
Rankin-Selberg convolution and use this convolution to produce a meromorphic
continuation of the continuous component of $D(s, w)$.

Explicit computation shows that
\begin{equation}\label{eq:cont_ds_is_RS}
\begin{split}
  &\big\langle
    E_\mathfrak{a}(\cdot, \tfrac{1}{2} + it; 2N) {\Im(\cdot)}^\frac{k}{2}
    \overline{f_1(\cdot)}
    ,
    E_\infty^k(\cdot, \overline{s})
  \big\rangle
  \\
  &\quad=
  \frac{2 \pi^{\frac{1}{2} + it}} {\Gamma(\frac{1}{2} + it)}
  \sum_{h \geq 1}
  \frac{a(h) \rho_{\mathfrak{a}}(\frac{1}{2} + it, h)}
       {h^{-it}}
  \int_0^\infty
  y^{s + \frac{k}{2} - \frac{1}{2}}
  K_{it}(2\pi h y) e^{-2\pi h y} \frac{dy}{y}
  \\
  &\quad=
  \frac{2 \pi^{1 + it}} {{(4 \pi)}^{s + \frac{k}{2} - \frac{1}{2}}}
  \frac{\Gamma(s + \frac{k}{2} - \frac{1}{2} + it)
        \Gamma(s+\frac{k}{2} - \frac{1}{2} - it)}
       {\Gamma(\frac{1}{2}+ it)
        \Gamma(s+\frac{k}{2})}
  \sum_{h \geq 1}
  \frac{a(h) \rho_{\mathfrak{a}}(\frac{1}{2}+it,h)}
       {h^{s-it+\frac{k}{2}-\frac{1}{2}}}.
\end{split}
\end{equation}
It follows that the continuous spectrum can be written as
\begin{equation}\label{eq:continuous_ref}
\begin{split}
  &\frac{2^{3s + 2w + 3k - 5} \pi^{s + w + k - 2}}{i \Gamma(s + k - 1)}
  \sum_{\mathfrak{a}}
  \int_{-i\infty}^{i\infty}
  \frac{\Gamma(s + w)
        \Gamma(s - \frac{1}{2} + z)
        \Gamma(s - \frac{1}{2} - z)}
       {\Gamma(s)
        \Gamma(s + w - \frac{1}{2} + z)
        \Gamma(s + w - \frac{1}{2} - z)}
  \\
  &\quad \times
  \big\langle
     E_\mathfrak{a}(\cdot,\tfrac{1}{2}+ z; 2N)\Im{(\cdot)}^\frac{k}{2}
    \overline{f_1(\cdot)}
    ,
    E_\infty^k(\cdot,\overline{s + w - \tfrac{k}{2}})
  \big\rangle
  \big\langle
    V, E_\mathfrak{a}(\cdot, \tfrac{1}{2} - \overline{z}; 2N)
  \big\rangle dz.
\end{split}
\end{equation}

Each part in this expression for the continuous spectrum has a clear
meromorphic continuation, but the integral entangles poles in $s$ with
those of $z$.
Considering the poles of the gamma functions and Eisenstein series, it is
immediately clear that the continuous component has meromorphic continuation to
the region defined by $\Re(s + w) - \frac{k}{2} > \frac{1}{2}$ and
$\Re s > \frac{1}{2}$.
This is sufficient to prove the primary theorems in the next section.
But we also explore how to delicately and iteratively extend the meromorphic
continuation of the continuous spectrum by carefully shifting the line of
integration and collecting residual terms.

Initially take $\Re w$ large.
For small $\epsilon > 0$, let $\Re s$ be in the interval $(\frac{1}{2},
\frac{1}{2} + \epsilon)$.
Shift the $z$-contour of integration to the right along a contour $C$ which
bends to remain in the zero-free region of $\zeta(1 - 2z)$, avoiding the
potential poles from these zeroes in
$E_\mathfrak{a}(\cdot, \frac{1}{2} - \overline{z}; 2N)$.
Taking $\epsilon$ sufficiently small, this shift of contour passes a pole at
$z = s - \frac{1}{2}$ with residual term
\begin{equation}
\begin{split}
  \mathcal{R}^-
  &=
  \frac{2^{3s + 2w + 3k - 4} \pi^{s + w + k - 1}}
       {\Gamma(s + k - 1) \cdot 2\pi i}
  \sum_{\mathfrak{a}}
  \frac{\Gamma(s + w)\Gamma(2s - 1)}
       {\Gamma(s)\Gamma(2s + w - 1)\Gamma(w)}
  \\
  &\quad \times
  \big\langle  E_\mathfrak{a}(\cdot,s; 2N)\Im{(\cdot)}^\frac{k}{2}
    \overline{f_1(\cdot)}
    ,
    E_\infty^k(\cdot,\overline{s + w - \tfrac{k}{2}})
  \big\rangle
  \left\langle V, E_\mathfrak{a}(\cdot, 1 - \overline{s}; 2N)\right\rangle.
\end{split}
\end{equation}
Note that the residual term $\mathcal{R}^-$ has clear meromorphic continuation
to $\mathbb{C}^2$ and is analytic in the region defined by
$\Re(s + w) - \frac{k}{2} > \frac{1}{2}$ and $\Re s > 0$, except for a potential
pole at $s = \frac{1}{2}$ from $\Gamma(2s - 1)$.


The deformation of the contour integral in~\eqref{eq:continuous_ref} along the
contour $C$ is analytic for $\Re(s)$ to the right of the contour
$\frac{1}{2} - C$ and to the left of the line $\frac{1}{2} + \epsilon$.
Examining a value of $s$ with real part left of the line $\frac{1}{2}$ but still
within the region to the right of the contour $\frac{1}{2} - C$, we can deform
the contour back to the line $\Re z = 0$.
This passes the pole at $z = \frac{1}{2} - s$ from the other gamma
function, giving the residual term
\begin{equation}
\begin{split}
  \mathcal{R}^+
  &=
  \frac{2^{3s + 2w + 3k - 4}\pi^{s + w + k - 1}}
       {\Gamma(s + k - 1) \cdot 2\pi i}
  \sum_{\mathfrak{a}}
  \frac{\Gamma(s + w)\Gamma(2s - 1)}
       {\Gamma(s)\Gamma(w)\Gamma(2s + w - 1)}
  \\
  &\quad \times
  \big\langle
     E_\mathfrak{a}(\cdot,1-s; 2N)\Im{(\cdot)}^\frac{k}{2}
    \overline{f_1(\cdot)}
    ,
    E_\infty^k(\cdot,\overline{s + w - \tfrac{k}{2}})
  \big\rangle
  \left\langle V, E_\mathfrak{a}(\cdot, \overline{s}; 2N)\right\rangle.
\end{split}
\end{equation}


For $\Re w$ sufficiently large, the now un-deformed contour integral
in~\eqref{eq:continuous_ref} is analytic for $s$ with $-\frac{1}{2} < \Re s <
\frac{1}{2}$.
Thus the continuous spectrum has meromorphic continuation to the region
$\Re(s + w) - \frac{k}{2} > \frac{1}{2}$ and $\Re s > -\frac{1}{2}$, and the
only poles in this region occur in the two residual terms $\mathcal{R}^-$ and
$\mathcal{R}^-$.

As in~\cite[\S4, p. 481--483]{Jeff} or~\cite[\S4]{hkldw1}, it is possible to
iterate this argument: for each pair of conflated poles in $s$ and $z$, one can
shift and unshift the contour of integration to extend the region of meromorphy
at the cost of introducing additional residual terms.
Each residual term has clear meromorphic continuation to $\mathbb{C}^2$, and
thus so does the continuous component.

\begin{remark}
The authors have employed this iterative technique of disambiguating  poles in
appearing in the continuous spectrum several times in the past after
specializing to the case where $w = 0$.
In those cases, the pair of residual terms $\mathcal{R}^-$ and
$\mathcal{R}^+$ are anti-symmetric.
\end{remark}

\subsection{Polar behavior of $D(s,w)$}

Having described the meromorphic continuation of the discrete and continuous
parts of $D(s, w)$, we now summarize the polar behavior of $D(s, w)$ necessary
for the proof of the theorems in the next section.

\begin{theorem}\label{thm:D(s,w)_poles}
  The multiple Dirichlet series $D(s, w)$ has meromorphic continuation to
  $\mathbb{C}^2$.
  For $\Re s  + \Re w > \frac{k}{2}$ and $\Re s > 0$, $D(s,w)$ has potential
  poles at
  \begin{enumerate}
    \item $s=\frac{1}{2} \pm it_j - r$, where $r$ is a nonnegative integer,
      arising from $\Gamma(s - \frac{1}{2} \pm it_j)$
      in~\eqref{eq:discrete_spectrum2}
    \item $2s + 2w - k = \rho$, where $\rho$ is a zero of $\zeta^{(2N)}(s)$,
      arising from $\zeta^{(2N)}(2s + 2w - k)$ in~\eqref{eq:discrete_spectrum2}
      or $E_\infty^k(\cdot, s + w - \frac{k}{2})$ in~\eqref{eq:continuous_ref}
    \item $s = \frac{1}{2}$
      arising from $\Gamma(2s - 1)$ in the residual terms $\mathcal{R}^-$ and
      $\mathcal{R}^+$
  \end{enumerate}
\end{theorem}

\section{Bounds on Triple Correlation Sums}

In this section we consider a double integral transform of the form
\begin{equation}\label{eq:reference_integral_transform}
\begin{split}
  \int_{\sigma_w -i\infty}^{\sigma_w + i\infty}
  &\int_{\sigma_s -i\infty}^{\sigma_s + i\infty}
  D(s, w)
  X^{s+k-1} Y^w
  \Gamma(s + k - 1) \Gamma(w)
  \, ds \, dw
  \\
  &=
  \sum_{h, m \geq 1}
  a(h) b(m) c(2m - h) e^{-m/X} e^{-h/Y}
\end{split}
\end{equation}
in order to prove our theorems concerning the sizes of the triple correlation
sums.
Initially, we take the lines of integration to be $\Re s > 1$ and $\Re w >
\frac{k+1}{2}$, within the domain of absolute convergence of $D(s, w)$.

Our main theorem follows quickly from the meromorphic description of $D(s, w)$
and from recognizing that $D(s, w)$ grows at most polynomially in vertical
strips.

\subsection{Polynomial growth}

We now examine the discrete component~\eqref{eq:discrete_spectrum2}.
Let $\sigma_s := \Re s$ and $\sigma_w := \Re w$.
Stirling's approximation demonstrates that the exponential contribution from the
gamma factors is
\begin{equation}
  \exp\big(
    \pi \lvert \Im s \rvert
    -
    \pi \max( \lvert t_j \rvert, \lvert \Im s \rvert)
  \big),
\end{equation}
wherein we've used that $\lvert a + b \rvert + \lvert a - b \rvert = 2
\max(a, b)$ to simplify.
There is no other source of exponential growth.
Thus for $\lvert t_j \rvert > \lvert \Im s \rvert^{1+\epsilon}$, the sum over $t_j$ has
exponential decay and rapidly converges; this is quickly seen to not be the
dominant contribution.
For $\lvert t_j \rvert \leq \lvert \Im s \rvert^{1+\epsilon}$, there is no exponential
contribution, and more care must be given.

It follows from Stirling's approximation that the polynomial contribution from
the gamma factors is
\begin{equation}
  \frac{%
    {(1 + \lvert \Im s + t_j \rvert)}^{\sigma_s - 1}
    {(1 + \lvert \Im s - t_j \rvert)}^{\sigma_s - 1}
  }
  {
    {(1 + \lvert \Im s \rvert)}^{2\sigma_s + k - 2}
  }.
\end{equation}
In the region $\lvert t_j \rvert \leq \lvert \Im s \rvert^{1+\epsilon}$, this is clearly of
polynomial growth in $\lvert \Im s \rvert$.

To understand the rest of the $j$-sum, it is necessary to decouple the growth in
$t_j$ from the $L$-function.
Writing each $L$-function as $L(s, \mu_j \otimes \overline{f_1}) =
\rho_j(1) \widetilde{L}(s, \mu_j \otimes \overline{f_1})$ effectively
focuses the $j$-dependence into the $\rho_j(1)$ term, while classical Hecke and
convexity bounds can handle $\widetilde{L}(s, \mu_j \otimes \overline{f_1})$.

In particular, the convexity bound, the functional
equation~\eqref{equation:discrete_functional_equation}, and the classical bound
$1/\zeta(1 + z) \ll \lvert \log z \rvert$~\cite[3.11.10]{titchmarsh1986theory}
show that for $\Re (s + w) \geq \frac{k}{2}$, we have the bound
\begin{equation*}
  \frac{\widetilde{L}(s + w - \frac{k}{2}, \mu_j \otimes \overline{f_1})}
  {\zeta(2s + 2w - k)}
  \ll
  {(1 + \lvert \Im s + \Im w \rvert)}^{2 + \epsilon}
\end{equation*}
for any $\epsilon > 0$. More generally, away from poles, this is of polynomial
growth in $\lvert \Im s \rvert$ and $\lvert \Im w \rvert$ in vertical strips.

Bounding by absolute values and applying Lemma~\ref{lem:jeff_bound} to the sum
over those $t_j$ with $\lvert t_j \rvert \leq \lvert \Im s \rvert^{1+\epsilon}$, it follows
that, away from poles, the discrete component is of polynomial growth in $\lvert
\Im s \rvert$ and $\lvert \Im w \rvert$ in vertical strips.

The continuous component~\eqref{eq:continuous_sum_ref} is very similar and gives
almost the same bound.
The four gamma factors containing an $s$ are identical to the four appearing in
the discrete spectrum, except with $t$ in place of $t_j$; correspondingly
the analysis with Stirling's formula carries over. There is rapid exponential
decay when $\lvert t \rvert > \lvert \Im s \rvert^{1+\epsilon}$, and polynomial growth
otherwise.

The convexity bound shows that, for $\Re (s + w) > \frac{k}{2}$,
\begin{equation*}
  \Big\lvert
  \sum_h
  \frac{a(h)\rho_\mathfrak{a}(\frac{1}{2} + it, h)}
       {h^{s + w - \frac{1}{2} - it}}
  \Big\rvert
  \ll
  {\big(1 + \lvert \Im s + \Im w \rvert + \lvert t \rvert\big)}^{2 + \epsilon}
\end{equation*}
for any $\epsilon > 0$. We note that to get this convexity bound, we regard this
Dirichlet series as being of Rankin-Selberg type. One can use the functional
equation of the Eisenstein series in~\eqref{eq:cont_ds_is_RS} to understand the
functional equation of this Dirichlet series. More generally, away from poles it
is clear that this Rankin-Selberg type Dirichlet series has polynomial growth in
vertical strips.

Finally, combining these bounds together with Lemma~\ref{lem:jeff_bound_cont},
we see that the integral has exponential decay and doesn't contribute
meaningfully for $\lvert t \rvert > \lvert \Im s \rvert^{1+\epsilon}$, and otherwise is of
polynomial growth in $\lvert \Im s \rvert$ and $\lvert \Im w \rvert$.

Thus the continuous component is of polynomial growth in vertical strips, away
from poles.

\subsection{Proofs of Main Results}

We first prove a result analogous to simultaneous square-root cancellation
in each variable.

\begin{theorem}\label{thm:full_main}
  Let $a(\cdot)$ denote the coefficients of a holomorphic cuspidal eigenform of
  weight $k$, level $N$, and trivial nebentypus. Then for any $\epsilon > 0$,
  \begin{equation*}
    \sum_{m, h \geq 1}
    a(h)b(m)c(2m-h) e^{-m/X} e^{-h/Y}
    \ll
    X^{k-1 + \Theta + \frac{1}{2} + \epsilon}
    Y^{\frac{k-1}{2} + \frac{1}{2} - \Theta + \epsilon},
  \end{equation*}
  where $\Theta<7/64$ is the best bound towards the non-holomorphic
  Ramanujan-Petersson conjecture.
\end{theorem}

\begin{proof}

Consider the inverse Mellin transform~\eqref{eq:reference_integral_transform}.
By Theorem~\ref{thm:D(s,w)_poles}, $D(s,w)$ is holomorphic when
$\Re s > \frac{k+1}{2}-\Re w$ and $\Re s> \frac{1}{2} \pm \max(\Re t_j)$, where
$t_j$ ranges over the types of the Maass forms in the discrete spectrum.
Shifting lines to $\Re s = \frac{1}{2} + \Theta+\epsilon$ and
$\Re w = \frac{k-1}{2} + \frac{1}{2} - \Theta + \epsilon$ avoids all poles.
The shifted integral clearly converges since $D(s, w)$ is of polynomial growth
in $\lvert \Im s \rvert$ and $\lvert \Im w \rvert$
while $\Gamma(s+k-1)\Gamma(w)$ have exponential decay in $\lvert \Im s \rvert$
and $\lvert \Im w \rvert$.

\end{proof}

Assuming the Riemann Hypothesis, it is possible to further shift the $w$
variable by an additional $1/4$ before encountering the poles from the zeta
functions in the denominator. Thus we also have the following theorem.

\begin{theorem}
  Assume the Riemann Hypothesis.
  Using the same notation as in Theorem~\ref{thm:full_main}, for any $\epsilon >
  0$ we have
  \begin{equation*}
    \sum_{m, h \geq 1}
    a(h)b(m)c(2m-h) e^{-m/X} e^{-h/Y}
    \ll
    X^{k - 1 + \frac{1}{2} + \Theta + \epsilon}
    Y^{\frac{k-1}{2} + \frac{1}{4}-\Theta + \epsilon}.
  \end{equation*}
\end{theorem}

\section{Nonvanishing result}

We now prove that, under mild hypotheses, infinitely many products
$a(h)b(m)c(2m-h)$ do not vanish. To do this, we specialize $w$ and examine under
what conditions the residues of poles coming from the discrete spectrum do not
vanish. We note that it would also be possible to consider poles corresponding
to zeros of $\zeta^{(2N)}$, occurring in the continuous spectrum.

We fix $w = \frac{k}{2} + \frac{3}{2}$ and examine potential poles with $\Re s >
0$. This guarantees that $\Re s + \Re w > \frac{k}{2}$.
Theorem~\ref{thm:D(s,w)_poles} indicates that there are potential poles at
$s = \frac{1}{2} \pm it_j$ occurring from the discrete spectrum.

From the description of the discrete spectrum in Lemma~\ref{lem:disc}, we
compute that the residue at $s = \frac{1}{2} \pm it_j$ is
\begin{equation}
  \frac{{(8\pi)}^{k - \frac{1}{2} \pm it_j}}{\Gamma(k - \frac{1}{2} \pm it_j)}
  \frac{\sqrt{\pi}}{{(4\pi)}^{\pm it_j}}
  \frac{\Gamma(\pm 2 i t_j)}{\Gamma(\frac{1}{2} \pm it_j)}
  \langle V, \mu_j \rangle
  \frac{L(2 \pm it_j, \mu_j \otimes \overline{f_1})}
  {\zeta^{(2N)}(4 \pm 2 it_j)}.
\end{equation}
The ratios of Gamma functions and powers of $2$ and $\pi$ are some nonzero
constant $C_j$. It remains only to consider the inner product $\langle V,
\mu_j\rangle$, the $L$-function, and the $\zeta$-function.

The zeta function is considered far within the region of absolute convergence,
and thus is evaluated far from poles and zeros. Similarly, the
Rankin--Selberg 
convolution $L(2 \pm it_j, \mu_j \otimes \overline{f_1})$ is evaluated within
its domain of absolute convergence. As both $L(s, \mu_j)$ and $L(s, f_1)$ have
Euler products, the general theory of Rankin--Selberg 
convolutions guarantees that $L(s, \mu_j \otimes \overline{f_1})$ has an Euler
product (see for instance Chapter~12 of~\cite{Goldfeld2006automorphic}). As no
factor of the Euler product is zero within the domain of absolute convergence,
we see that $L(2 \pm it_j, \mu_j \otimes \overline{f_1}) \neq 0$.
Thus every factor in the residue is nonzero with the possible exception of
$\langle V, \mu_j \rangle$.

Let us suppose that there is a Maass form $\mu_j$ such that $\langle V, \mu_j
\rangle \neq 0$ and such that $t_j \neq 0$ is real. Let us fix that form
$\mu_j$.

\begin{remark}
Note that the Laplacian eigenvalue of $\mu_j$ is $\lambda = (\frac{1}{2} +
it_j)(\frac{1}{2} - it_j)$ and that $\lambda$ is real and nonnegative. Thus
$it_j$ is either purely real or purely imaginary. If $it_j$ is purely real, then
the Maass form is \emph{exceptional}. The condition that $t_j \neq 0$ is real is
thus equivalent to $\lambda > \frac{1}{4}$.
\end{remark}

The Dirichlet series
\begin{equation}
  D(s, \tfrac{k}{2} + \tfrac{3}{2})
  =
  \sum_{m \geq 1} \Bigg(
    \sum_{h \geq 1} \frac{a(h) c(2m - h)}{h^{\frac{k}{2} + \frac{3}{2}}}
  \Bigg) \frac{b(m)}{m^{s}}
\end{equation}
can be regarded as a single Dirichlet series in $s$ with meromorphic
continuation to $\mathbb{C}$. Note that summing over $h \geq 1$ is equivalent to
summing over $h \leq 2X$, as for $h > 2X$ we have $c(2m - h) = 0$. As $\langle
V, \mu_j \rangle \neq 0$, we see that this Dirichlet series has a pair of poles
at $s = k - \frac{1}{2} \pm it_j$.

Then Theorem~1 of~\cite{dldomega} applies and yields the following result.

\begin{theorem}
  Let $\mu_j$ be a non-constant Maass form such that $\langle V, \mu_j \rangle
  \neq 0$, and let $\mathrm{MT}(X)$ denote the sum of the residues of $D(s,
  \tfrac{k}{2} + \tfrac{3}{2})X^s/s$ at all real poles $s = \sigma$ with $\sigma \geq
  k - \tfrac{1}{2}$. Then
  \begin{equation}
    \sum_{m \leq X} \sum_{h \leq 2X}
    \frac{a(h)b(m)c(2m-h)}
         {h^{\frac{k}{2} + \frac{3}{2}}}
  - \mathrm{MT}(X)
  = \Omega_\pm(X^{k - \frac{1}{2}})
  \end{equation}
\end{theorem}

This is a more precise statement of Theorem~\ref{thm:omega}.
As an immediate corollary, it follows that infinitely many triples
$a(h)b(m)c(2m-h)$ are non-vanishing.

\begin{remark}
The main term $\mathrm{MT}(X)$ consists of residues at exceptional eigenvalues
and the residue at $s = k - \tfrac{1}{2}$. In cases where we expect Selberg's
Eigenvalue Conjecture to hold, the main term will arise entirely out of the
pole at $s = k - \tfrac{1}{2}$.
\end{remark}

\bibliographystyle{alpha}
\bibliography{jobbib}

\begin{thebibliography}{HKLDW18b}

\bibitem[BH08]{blomer2008spectral}
Valentin Blomer and Gergely Harcos.
\newblock The spectral decomposition of shifted convolution sums.
\newblock {\em Duke Mathematical Journal}, 144(2):321--339, 2008.

\bibitem[Blo04]{Blomer}
Valentin Blomer.
\newblock Shifted convolution sums and subconvexity bounds for automorphic
  {$L$}-functions.
\newblock {\em Int. Math. Res. Not.}, (73):3905--3926, 2004.
\newblock http://dx.doi.org/10.1155/S1073792804142505.

\bibitem[Blo17]{blomer2017triple}
Valentin Blomer.
\newblock On triple correlations of divisor functions.
\newblock {\em Bulletin of the London Mathematical Society}, 49(1):10--22,
  2017.

\bibitem[CN62]{Chand}
K.~Chandrasekharan and Raghavan Narasimhan.
\newblock Functional equations with multiple gamma factors and the average
  order of arithmetical functions.
\newblock {\em Ann. of Math. (2)}, 76:93--136, 1962.

\bibitem[Con08]{conrad}
Keith Conrad.
\newblock The congruent number problem.
\newblock {\em Harvard College Mathematical Review}, 2:58--74, 2008.
\newblock http://www.math.harvard.edu/hcmr/issues/2a.pdf.

\bibitem[Gol06]{Goldfeld2006automorphic}
Dorian Goldfeld.
\newblock {\em Automorphic forms and {L}-functions for the group GL (n, R)},
  volume~13.
\newblock Cambridge University Press, 2006.

\bibitem[GR15]{GradshteynRyzhik07}
I.~S. Gradshteyn and I.~M. Ryzhik.
\newblock {\em Table of integrals, series, and products}.
\newblock Elsevier/Academic Press, Amsterdam, eighth edition, 2015.
\newblock Translated from the Russian, Translation edited and with a preface by
  Daniel Zwillinger and Victor Moll, Revised from the seventh edition
  [MR2360010].

\bibitem[HB79]{HB79}
D.~R. Heath-Brown.
\newblock The fourth power moment of the {R}iemann zeta function.
\newblock {\em Proc. London Math. Soc. (3)}, 38(3):385--422, 1979.

\bibitem[HH16]{Jeff}
Jeff Hoffstein and Thomas~A. Hulse.
\newblock Multiple {D}irichlet series and shifted convolutions, with an
  appendix by {A}ndre {R}eznikov.
\newblock {\em J. Number Theory}, 161:457--533, 2016.
\newblock https://dx.doi.org//10.1016/j.jnt.2015.10.001.

\bibitem[HKLDW17]{hkldw1}
Thomas~A. Hulse, Chan~Ieong Kuan, David Lowry-Duda, and Alexander Walker.
\newblock The second moment of sums of coefficients of cusp forms.
\newblock {\em Journal of Number Theory}, 173:304--331, 2017.

\bibitem[HKLDW18a]{HKLW4}
Thomas~A. Hulse, Chan~Ieong Kuan, David Lowry-Duda, and Alexander Walker.
\newblock Second moments in the generalized {G}auss circle problem.
\newblock 2018.

\bibitem[HKLDW18b]{hkldw_congruent}
Thomas~A. Hulse, Chan~Ieong Kuan, David Lowry-Duda, and Alexander Walker.
\newblock A shifted sum for the congruent number problem, 2018.

\bibitem[IK04]{IK}
Henryk Iwaniec and Emmanuel Kowalski.
\newblock {\em Analytic number theory}, volume~53 of {\em American Mathematical
  Society Colloquium Publications}.
\newblock American Mathematical Society, Providence, RI, 2004.

\bibitem[Ivi03]{ivic2003note}
A~Ivic.
\newblock A note on the laplace transform of the square in the circle problem.
\newblock {\em Studia Sci. Math. Hung.}, 37(math. NT/0312255):391--399, 2003.

\bibitem[LD19]{dldomega}
David Lowry-Duda.
\newblock Non-real poles and irregularity of distribution.
\newblock Available as an arXiv preprint: arXiv:1910.09969, 2019.

\bibitem[Lin18]{lin2018triple}
Yongxiao Lin.
\newblock Triple correlations of fourier coefficients of cusp forms.
\newblock {\em The Ramanujan Journal}, 45(3):841--858, 2018.

\bibitem[Sar01]{Sarnak}
Peter Sarnak.
\newblock Estimates for {R}ankin-{S}elberg {$L$}-functions and quantum unique
  ergodicity.
\newblock {\em J. Funct. Anal.}, 184(2):419--453, 2001.
\newblock http://dx.doi.org/10.1006/jfan.2001.3783.

\bibitem[Sin18]{singh2018double}
Saurabh~Kumar Singh.
\newblock On double shifted convolution sum of sl (2, z) hecke eigenforms.
\newblock {\em Journal of Number Theory}, 2018.

\bibitem[THB86]{titchmarsh1986theory}
Edward~Charles Titchmarsh and David~Rodney Heath-Brown.
\newblock {\em The theory of the {R}iemann zeta-function}.
\newblock Oxford University Press, 1986.

\end{thebibliography}

\end{document}